\documentclass[12pt,final]{amsart}

\usepackage{graphicx} 
\usepackage[margin=1in]{geometry}
\usepackage{amsmath}
\usepackage{hyperref}
\usepackage{verbatim}
\usepackage{amssymb}
\usepackage{amscd}
\usepackage{xcolor}

\newtheorem{lemma}{Lemma}[section]

\newtheorem{theorem}[lemma]{Theorem}

\theoremstyle{definition}

\theoremstyle{remark}
\newtheorem{remark}[lemma]{Remark}

\newcommand{\Z}{\ensuremath{{\mathbb Z}}}
\newcommand{\N}{\ensuremath{{\mathbb N}}}
\newcommand{\Q}{\ensuremath{\mathbb Q}}
\newcommand{\R}{\ensuremath{\mathbb R}}
\newcommand{\C}{\ensuremath{\mathbb C}}

\newcommand{\tr}{\ensuremath{\mathrm {tr}}}
\newcommand{\cS}{\ensuremath{\mathcal {S}}}

\newcommand{\tU}{\ensuremath{\Tilde{U}}}
\newcommand{\tV}{\ensuremath{\Tilde{V}}}
\newcommand{\tu}{\ensuremath{\Tilde{u}}}
\newcommand{\tv}{\ensuremath{\Tilde{v}}}
\newcommand{\Sph}{\ensuremath{{\mathbb{S}^{2mn-1}}}}
\newcommand{\Bm}{\ensuremath{\mathbb{B}^{m}}}

\author{Tatyana Barron}
\address{Department of Mathematics, University of Western Ontario, London Ontario N6A 5B7, Canada}
\email{tatyana.barron@uwo.ca}

\author{Manimugdha Saikia}
\address{Department of Mathematics, University of Western Ontario, London Ontario N6A 5B7, Canada}
\email{msaikia@uwo.ca}

\title{Average entropy and asymptotics}
\begin{document}

\thanks{This work was partially financially supported by the University of Western Ontario Bridge Grant project number R3659A04.}

\subjclass[2020]{Primary 53D50, 81S10, 81P17}
\keywords{K\"ahler manifold, line bundle, semiclassical limit, entanglement entropy}

\begin{abstract}
We determine the $N\to\infty$ asymptotics of the expected value of entanglement entropy for pure states in $H_{1,N}\otimes H_{2,N}$, where 
$H_{1,N}$ and $H_{2,N}$ are the spaces of holomorphic sections of the $N$-th tensor powers of hermitian ample line bundles on compact complex manifolds.  
\end{abstract}

\maketitle

\section{Introduction} 

There are various mathematical notions of entropy. It quantifies chaos, mixing, disorder or complexity. The Shannon entropy used in information theory has a probabilistic interpretation or it can be viewed as a way to quantify information. The Shannon entropy of a probability measure $P_X$ on a finite set 
$X=\{ x_1,...,x_r\}$ with masses $\{ p_1,...,p_r\} $ ($p_j=P_X(x_j)$, $1\le j\le r$)  equals $-\sum\limits _{j=1}^r p_j \ln p_j$, with the convention $p_j\ln p_j=0$ when $p_j=0$.  
Let $H_1$ and $H_2$ be finite-dimensional Hilbert spaces. The partial trace $\tr_2$ is the linear map $\tr_2: $ End$(H_1 \otimes H_2) \to $ End$(H_1)$
  given by $\tr_2(A\otimes B) = \tr(B)A$ and extended by linearity.
The entanglement entropy $E(v)$ of a vector $v\in H_1\otimes H_2$ is 
$E(v)=-\sum_{j=1}^m\lambda_j \ln \lambda_j$, where $\lambda_1$,...,$\lambda_m$ are the eigenvalues of $\tr_2(P_v)$,
the linear map $P_v$ is the orthogonal projection from $H_1\otimes H_2$ onto the one-dimensional linear subspace spanned by $v$, 
 and as before we use 
the convention $0\ln 0=0$.  Note:  $P_v=vv^*$. 
The vector $v$ is decomposable if and only if $E(v)=0$. Calculations of entropy on the Hilbert spaces of geometric quantization or Toeplitz quantization lead to interesting insights \cite{barronp, charles}. 
In \cite{zelditchf}, the main result is  the $k\to\infty$ asymptotics of the Shannon entropies of $\mu_z^k$,  where $k\in\N$, $z\in M$, $M$ is a toric K\"ahler manifold with an ample toric hermitian line bundle, and $\mu_z^k$ are the Bergman measures that were introduced by Zelditch in \cite{zeld2009} to define generalized Bernstein polynomials and were subsequently used in \cite{songzeld2012, zeldzhou2021}. 
In a series of papers on random sections of line bundles, starting  with \cite{shifzeld}, Shiffmann and Zelditch worked with the probability space 
 $$
 \prod_{k=1}^{\infty}SH^0(M,L^k)
 $$
 where $L\to M$ is an ample holomorphic hermitian line bundle on a compact complex manifold $M$ and $SH^0(M,L^k)$ is the unit sphere in the finite-dimensional Hilbert  space 
 $H^0(M,L^k)$. 
 In this paper, we consider instead the probability space   
\begin{equation}
\label{productof2}
\Omega= \prod_{k=1}^{\infty}S(H^0(M,L^k)\otimes H^0(M,L^k)). 
 \end{equation}
and a sequence of random variables ($\R$-valued functions on $\Omega$) $E_k\circ p_k$, where $p_k$ is the projection to the $k$-th component in the product $\prod_{k=1}^{\infty}$ above in (\ref{productof2}), and $E_k$ is the entanglement entropy. We find the $k\to\infty$ asymptotics of the sequence of expected values of these random variables. In fact, we prove a more general result.   
\begin{theorem}
\label{mainth}
Let $L_1\to M_1$ and $L_2\to M_2$ be positive holomorphic hermitian line bundles on compact complex manifolds $M_1$ and $M_2$ of complex dimensions $d_1$ and $d_2$ respectively.  Assume w.l.o.g. $d_1\le d_2$. 
Let $d\mu_N$, for each $N\in \N$, be the measure on the unit sphere $S_N=S(H^0(M_1,L_1^N)\otimes H^0(M_2,L_2^N))$ induced by the hermitian metrics. 
There are the following $N\to\infty$ asymptotics for the average entanglement entropy 
$$
\langle E _N\rangle =\frac {\int_{S_N} E_N(v) d\mu_N(v)} { \int_{S_N} d\mu _N(v) } .
$$
Let 
$$
\beta_j=\int_{M_j} \frac{c_1(L_j)^{d_j}}{d_j!}
$$
$$
\gamma_j=\frac{1}{2}\int_{M_j} \frac{c_1(TM_j)c_1(L_j)^{d_j-1}}{(d_j-1)!}
$$
for $j\in \{ 1,2\}$. As $N\to\infty$, 
$$
\langle E_N\rangle\sim 
\left\{
                \begin{array}{ll}
                \ln\beta_1+d_1\ln N- \frac{\beta_1}{2\beta_2}  +\Bigl (\frac{\gamma_1}{\beta_1}  -\frac{\beta_1}{2\beta_2} (\frac{\gamma_1}{\beta_1}-\frac{\gamma_2}{\beta_2} )\Bigr ) \frac{1}{N}+O(\frac{1}{N^2}), \  {\mathrm{if}} \ d_1=d_2; \\
          \ln\beta_1+d_1\ln N+\Bigl ( \frac{\gamma_1}{\beta_1}-\frac{\beta_1}{2\beta_2}\Bigr )   \frac{1}{N}+O(\frac{1}{N^2}), \  {\mathrm{if}} \ d_1=d_2-1; \\
                 \ln\beta_1+d_1\ln N+\frac{\gamma_1}{\beta_1}\frac{1}{N}+O(\frac{1}{N^2}), \  {\mathrm{if}} \ d_1-d_2\le -2.
                \end{array}
              \right.
$$
\end{theorem}
\begin{remark}
We observe that a statement analogous to Theorem \ref{mainth} holds for semipositive line bundles $L_j$ on Moishezon manifolds $M_j$, $j\in \{1,2\}$. 
The following is true. 
Let $L_1\to M_1$ and $L_2\to M_2$ be holomorphic hermitian line bundles on compact connected complex manifolds $M_1$ and $M_2$ of complex dimensions $d_1$ and $d_2$ respectively.  Assume w.l.o.g. $d_1\le d_2$. Assume $M_1$ and $M_2$ are Moishezon and $L_1, L_2$ are semipositive. 
Let $d\mu_N$, for each $N\in \N$, be the measure on the unit sphere $S_N=S(H^0(M_1,L_1^N)\otimes H^0(M_2,L_2^N))$ induced by the hermitian metrics. 
There are the following $N\to\infty$ asymptotics for the average entanglement entropy on the Hilbert spaces 
$H^0(M_1,L_1^N)\otimes H^0(M_2,L_2^N)$:  as $N\to\infty$ 
$$
\langle E_N\rangle\sim 
\left\{
                \begin{array}{ll}
                \ln\beta_1+d_1\ln N- \frac{\beta_1}{2\beta_2} +o(1) , \  {\mathrm{if}} \ d_1=d_2; \\
                 \ln\beta_1+d_1\ln N +o(1), \  {\mathrm{if}} \ d_1<d_2.
                \end{array}
              \right.
$$
where, as before, 
$\beta_j=\int_{M_j} \frac{c_1(L_j)^{d_j}}{d_j!}$
for $j=1,2$. The proof is similar to the proof of Theorem \ref{mainth} in section \ref{mainproof} below, with (\ref{dimas}, \ref{dimas2}) replaced by  (from Th. 1.7.1 \cite{mmbook}) 
$$
m=m(N)=\dim H^0(M_1,L_1^N)=N^{d_1}\int_{M_1}\frac{c_1(L_1)^{d_1}}{d_1!}+o(N^{d_1})
$$   
$$
n=n(N)=\dim H^0(M_2,L_2^N)=N^{d_2}\int_{M_2}\frac{c_1(L_2)^{d_2}}{d_2!}+o(N^{d_2}).
$$  
\end{remark}

{\bf Acknowledgements.} We are thankful to the referee for helpful suggestions.

\section{Asymptotics} 
In this section, we establish the background and write the proofs needed for Theorem \ref{mainth}. An expression for the average entanglement entropy for the tensor product of two finite-dimensional Hilbert spaces is the statement of Page conjecture \cite{page}. There were several derivations of this formula in physics literature, 
including \cite{sen}.  
They assume the equality (\ref{eigenvformula}) (see below) as a starting point. Our Theorem \ref{proofeigenv} below is a proof of  (\ref{eigenvformula}). Then, our proof of Theorem \ref{averageth} follows the idea of Sen \cite{sen}. We rely on the semiclassical methods, together with the statement of  Theorem \ref{averageth}, to prove our main result, Theorem \ref{mainth} above.  

\subsection{Preliminaries}  
Let $H_1$ and $H_2$ be two complex Hilbert spaces of complex-dimension $m$ and $n$ respectively with $m\le n$. We note that $H_1\otimes H_2 \cong \C^m\otimes \C^n \cong \R^{2mn}$. Let $\Sph = \{v \in H_1\otimes H_2 : \|v \| =1 \} \subset H_1\otimes H_2$ be the unit sphere in $H_1\otimes H_2$ and $d\mu$ be the standard spherical measure on $\Sph$, normalized so that $\int_{\Sph}d\mu =1$. We fix an orthonormal basis $\{e_1,e_2,...,e_m\}$ for $H_1$ and $\{f_1,f_2,...,f_n\}$ for $H_2$. Then we have a canonical isomorphism $A: H_1 \otimes H_2 \to M_{n\times m}(\C)$ given by 
\begin{equation}\label{the A map}
    v = \sum_{j=1}^n\sum_{k=1}^m a_{jk}(v)e_k\otimes f_j \mapsto A(v)=(a_{jk}(v))_{n\times m}
\end{equation}
This map is also a diffeomorphism. Using $M_{n\times m}(\C) \cong \R^{2mn}$, we identify an $n\times m $ matrix with a point in $\R^{2mn}$. For $v\in H_1\otimes H_2$, we have the Singular Value Decomposition (SVD) of $A(v) = U(v)\begin{pmatrix}\Sigma(v) \\ 0\end{pmatrix} V(v)^*$, where $U(v)\in U(n)$ and $V(v)\in U(m)$ and $\Sigma(v) = \text{diag}\{\sigma_1(v),\sigma_2(v),...,\sigma_m(v)\}$ is a diagonal matrix with $0\le \sigma_1 \le \sigma_2 \le ... \le \sigma_m$, called the singular values of the matrix $A(v)$. Note that in the SVD of a matrix, the diagonal matrix $\Sigma(v)$ is unique (up to permutation), however the choice of $U(v)$ and $V(v)$ is not unique.\\

The Schmidt coefficients of $v$ are the same as the singular values of $A(v)$. If $\sigma_1(v),\sigma_2(v),...,\\ \sigma_ m(v)$ are Schmidt coefficients of $v \in \mathbb{B}^{2mn}$ (where $\mathbb{B}^{2mn}$ is the closed unit ball in $\R^{2mn}$), then $\sum_{j=1}^ m\sigma_j(v)^2 = \|v\|^2 \le 1$, i.e. for every $v\in \mathbb{B}^{2mn}$, there exists a triple $(U(v),\Sigma(v),V(v))$ such that $A(v) = U(v)\begin{pmatrix}\Sigma(v) \\ 0\end{pmatrix} V(v)^*$, where $U(v)\in U(n)$ and $V(v)\in U(m)$ are unitary matrices and $\Sigma(v) = \text{diag}\{\sigma_1(v),\sigma_2(v),...,\sigma_m(v)\}$ a diagonal matrix with $0\le \sigma_1(v) \le \sigma_2(v) \le ... \le \sigma_m(v)$ and  $\sum_{j=1}^m\sigma_j(v)^2\le1$. Conversely, if $A(v)\in M_{n\times m}$ has a SVD $A(v)= U(v)\begin{pmatrix}\Sigma(v) \\ 0\end{pmatrix} V(v)^*$ such that $\sum_{j=1}^m\sigma_j^2\le 1$, then $v \in \mathbb{B}^{2mn}$.\\

\begin{lemma}
Singular value decomposition $A=U\begin{pmatrix}\Sigma \\ 0\end{pmatrix} V^*$ of a matrix $A \in M_{n\times m}(\C)$ is not unique. However, for matrix $A$ with distinct singular values, if we put the condition that diagonal entries of $\Sigma$ are in the ascending order, the diagonal entries of the unitary matrix $V$ are non-negative and consider $U$ as an element of the complex Stiefel manifold $V_m^n(\C)(= U(n)/U(n-m))$, then this ``Modified SVD" is unique.
\end{lemma}

\begin{proof}\
Let $A$ be a $n\times m$ matrix with distinct singular values and $A = U\begin{pmatrix}\Sigma \\ 0\end{pmatrix} V^*$ be an SVD of $A$, so the entries of $\Sigma$ are distinct. Let $U= \begin{pmatrix} U_1 & U_2\end{pmatrix}$ where $U_1$ is the matrix whose columns are the first $m$ columns of $U$ and $U_2$ is the matrix whose columns are the last $n-m$ columns of $U$. Then we see that $A = U_1\Sigma V^*$. To see that SVD is not unique, let $\tau = \text{diag}\{e^{i\theta_1},e^{i\theta_2},...,e^{i\theta_m}\}$ for $-\pi < \theta_1,\theta_2,...,\theta_m \le \pi$. Let $V' = V\tau$ and $U'=U\begin{pmatrix}\tau & 0\\0 & I\end{pmatrix}$ are unitary matrices respective size. Then $U'\Sigma V'^* = \begin{pmatrix} U_1 & U_2\end{pmatrix}  \begin{pmatrix}\tau & 0\\0 & I\end{pmatrix}\begin{pmatrix}\Sigma\\0\end{pmatrix}(V\tau)^* = U_1\tau\Sigma\tau^*V^*= U_1\Sigma\tau\tau^*V^* = U_1\Sigma V^* =A$ is also a SVD of $A$.\\

Let $V = [v_{jk}] \in U(m)$. For each $j \in \{1,2,...,m\}$, we can write $v_{jj} = |v_{jj}|e^{i\theta_j}$ where $\theta_j \in (-\pi,\pi]$. Then
\begin{equation*}
    V = \begin{pmatrix} 
    |v_{11}| & v_{12}e^{-i\theta_2} & \dots & v_{1m}e^{-i\theta_m} \\
    v_{21}e^{-i\theta_1} & |v_{22}| & \dots & v_{2m}e^{-i\theta_m} \\
    \vdots & \vdots & \ddots & \vdots \\
    v_{m1}e^{-i\theta_1} & v_{m2}e^{-i\theta_2} & \dots & |v_{mm}|  \end{pmatrix}\begin{pmatrix} e^{i\theta_1} & 0 & \dots & 0 \\
    0 & e^{i\theta_2} & \dots & 0 \\
    \vdots & \vdots & \ddots & \vdots \\
    0 & 0 & \dots & e^{i\theta_m} 
    \end{pmatrix} \\
    = \tV \tau \text{ (say)}.   
\end{equation*}
Therefore, we get that any unitary matrix $V$ can be factored into a product of unitary matrix $\tV$ with diagonal entries being non-negative and a unitary diagonal matrix $\tau$. The factorization is unique and this gives us the one-to-one correspondence $U(m) \cong (U(m)/(U(1))^m) \times (U(1))^m$.\\

Now, suppose $A = U\begin{pmatrix}\Sigma \\ 0\end{pmatrix} V^*$ is an SVD of $A$ where diagonal entries of $\Sigma$ are in ascending order. Let $V = \tV\tau$ be the factorization of $V$ as above and $\tU =  \begin{pmatrix}U_1 & U_2\end{pmatrix} \begin{pmatrix}\tau^* & 0 \\ 0 & I\end{pmatrix}=  \begin{pmatrix}U_1\tau^* & U_2\end{pmatrix} =  \begin{pmatrix}\tU_1 & \tU_2\end{pmatrix}$ (say). Because of the uniqueness of the factors $\tV$ and $\tau$, the matrix $\tV$ is unique as an element of $U(m)/(U(1))^m$. The matrix $U_1$ is uniquely determined by $V$ (once $\Sigma$ is fixed to have ascending diagonal), so $\tU_1$ is also uniquely determined by $\tV$. Therefore, as an element of the complex Stiefel manifold $V_m^n(\C)$, we get a unique element $\tU \in V_m^n(\C)$ corresponding to $A$. Hence if we call $A = \tU \begin{pmatrix}\Sigma \\ 0\end{pmatrix}\tV^*$  the ``Modified SVD" of $A$, then  the modified SVD gives the unique triple $(\tU,\Sigma,\tV)$ with $\tU \in V_m^n(\C)$, $\Sigma \in \R^m_{\ge}$ (with ascending diagonal) and $\tV\in U(m)/(U(1))^m$.
\end{proof}

Note that for matrix $A$ with singular values that are not distinct, even the ``modified SVD" is not unique. This is because we can permute the equal singular values in $\Sigma$ to get a different SVD decomposition of $A$ where the columns of $\tU$ and $\tV$ are also permuted accordingly. 

\begin{theorem}
\label{proofeigenv}
Let $f$ be a continuous function on $\Sph$ that depends only on the squares of the Schmidt coefficients. For $v\in \mathbb{B}^{2mn}$, let $\sigma_1(v),...,\sigma_m(v)$ be the Schmidt coefficients of $v$ and $p_1(v),...,p_m(v)$ be the eigenvalues of $\tr_2(vv^*)$ (i.e. $p_j(v)=\sigma_{\tau(j)}^2(v)$ for some $\tau\in S_m$), then 
\begin{equation}
\label{eigenvformula}
\int\limits_{\Sph} f(u)d\mu(u)=
\end{equation}
$$
    \frac{\int_{T_m} f\left(\frac{p_1(v)}{\sum_{k=1}^mp_k(v)},...,\frac{p_m(v)}{\sum_{k=1}^mp_k(v)}\right)\prod\limits_{1\le j < k \le m}^m(p_k(v)-p_j(v))^2\prod_{j=1}^mp_j(v)^{n-m}dp_j(v)}{\int_{T_m}\prod_{1\le j < k \le m}^m(p_k(v)-p_j(v))^2 \prod_{j=1}^m p_j(v)^{n-m}dp_j(v)}
$$
where $T_m= \{(x_1,...,x_m) \in \R^m : x_1,...,x_m \ge 0 \text{ and } x_1+...+x_m \le 1\}$.
\end{theorem}

\begin{proof} 
We will use the Lebesgue measure $\nu$ of the ambient Euclidean space $\R^{2mn}$ and then for $X \subset \Sph$,
$\mu(X) = \frac{1}{\text{Vol}(\mathbb{B}^{2mn})}\nu(\{tx|x\in X \text{ and } t\in [0,1]\})$. Then for an integrable function defined on $\Sph$, we have
$$\int_{\Sph}f(u
)d\mu(u) = \frac{1}{\text{Vol}(\mathbb{B}^{2mn})}\int_{\mathbb{B}^{2mn}}f\left(\frac{v}{\|v\|}\right)d\nu(v).$$

Suppose $\cS = \{v\in \mathbb B^{2mn}: \text{ the singular values of $A(v)$ are distinct}\}$ where $A(v)$ is the matrix given by the map $A$ defined in Equation \eqref{the A map}. Let $\mathcal B$ be the set $\{(x_1,...,x_m) \in \R^m: 0\le x_1 < x_2 < ... < x_m \text{ and } x_1^2+...+x^2 \le 1\}$. Then the ``modified SVD" composed with the map $A$ gives an one-to-one correspondence between $\cS$ and $V_m^n(\C)\times \mathcal B  \times (U(m)/(U(1))^m)$. Note that $\mathbb B^{2mn} \setminus \cS$ is a set of measure zero, therefore
$$\int_{\mathbb{B}^{2mn}}f\left(\frac{v}{\|v\|}\right)d\nu(v) = \int_{\cS}f\left(\frac{v}{\|v\|}\right)d\nu(v).$$

For $v\in \cS$, we will make change of variables using ``modified SVD" $A(v) = \tU(v) \begin{pmatrix}\Sigma(v) \\ 0\end{pmatrix}\tV(v)^*$, where $\tU = \begin{pmatrix} \tU_1 & \tU_2 \end{pmatrix}$ with $\tU_1$ and $\tU_2$ the matrices with columns the first $m$ columns of $\tU$ and last $n-m$ columns of $\tU$ respectively. Let $\tu_j(v)$ be the $j$-th column of $\tU(v)$ and $\tv_k(v)$ be the $k$-th column of $\tV(v)$. Now,
$$
     A(v) = \tU(v) \begin{pmatrix}\Sigma(v) \\ 0\end{pmatrix}\tV^*(v) 
 $$
 hence
 $$   
    dA(v) = d\tU(v) \begin{pmatrix}\Sigma(v) \\ 0\end{pmatrix}\tV^*(v) + \tU(v) \begin{pmatrix}d\Sigma(v) \\ 0\end{pmatrix}\tV^*(v) + \tU(v) \begin{pmatrix}\Sigma(v) \\ 0\end{pmatrix}d\tV^*(v).
$$
The volume measure at $v_0 \in \cS$ can be written as $\tU^*(v_0) dA(v)\tV(v_0)|_{v=v_0}$ (as $\tU^*(v_0)$ and $\tV(v_0)$ being unitary matrices do not affect the volume). Let 
$$
dB(v_0) = \tU^*(v_0) dA(v)\tV(v_0)|_{v=v_0}.
$$ 
Then, 
\begin{align*}
    dB(v_0) & = \tU^*(v_0)d\tU(v_0)\begin{pmatrix}\Sigma(v_0) \\ 0\end{pmatrix} + \begin{pmatrix}d\Sigma(v_0) \\ 0\end{pmatrix} + \begin{pmatrix}\Sigma(v_0) \\ 0\end{pmatrix}d\tV^*(v_0)\tV(v_0)\\
    & = \begin{pmatrix} \tU_1^*(v_0)d\tU_1(v_0)\Sigma(v_0) + d\Sigma(v_0) + \Sigma(v_0)d\tV^*(v_0)\tV(v_0) \\ \tU_2^*(v_0)d\tU_1(v_0)\Sigma(v_0) \end{pmatrix}
\end{align*}

As $v_0$ was an arbitrary point in $\cS$, so we have 
\begin{align*}
   dB(v) & = \begin{pmatrix} \tU_1^*(v)d\tU_1(v)\Sigma(v) + d\Sigma(v) + \Sigma(v)d\tV^*(v)\tV(v) \\ \tU_2^*(v)d\tU_1(v)\Sigma(v) \end{pmatrix}
\end{align*}

For notational convenience, we drop the ``$(v)$" to write $$dB= \begin{pmatrix} \tU_1^*d\tU_1\Sigma + d\Sigma + \Sigma d\tV^*\tV \\ \tU_2^*d\tU_1\Sigma \end{pmatrix}.$$

Since $\tU_1^*\tU_1 = I$, it follows that $d\tU_1^*\tU_1 + \tU_1^*d\tU_1 = 0$, then 
$$
\tU_1^*d\tU_1 = - d\tU_1^*\tU_1 = -(\tU_1^*d\tU_1)^*,
$$  
so $\tU_1^*d\tU_1$ is an anti-hermitian matrix. Similarly, $\tV^*d\tV$ is also anti-hermitian. Denote $E = \tU_1^*d\tU_1$ and $F = \tV^*d\tV$, then for $j,k\in \{1,2,...,m\}$ we have $E_{jk} = \tu_j^*d\tu_k$ and $F_{jk} = \tv_j^*d\tv_k$. Now,
$$dB = \begin{pmatrix} E\Sigma + d\Sigma - \Sigma F \\ \tU_2^*d\tU_1\Sigma \end{pmatrix} .$$
As $E$ and $F$ are anti-hermitian, the diagonal elements of $E$ and $F$ are imaginary. It follows that the real parts of diagonal elements of $dB$ are exactly the diagonal elements of $d\Sigma$ and imaginary parts of the diagonal elements of $dB$ come from the matrix $E\Sigma - \Sigma F $, i.e, 
$$
\Re{(dB_{jj})} = d\Sigma_{jj} = d\sigma_j,
$$ 
$$
\Im{(dB_{jj})} = \sigma_j(\Im({\tu_j^*d\tu_j})-\Im({\tv_j^*d\tv_j})),
$$ 
for $j \in \{1,2,...,m\}$. Let $j \in \{1,2,...,m\}$ with $j>k$, then
\begin{align*}
    dB_{jk} & = \sigma_kE_{jk}-\sigma_jF_{jk},\\
    dB_{kj} & = \sigma_jE_{kj}-\sigma_kF_{kj}\\
    & = -\sigma_j\overline{E_{jk}}+\sigma_k\overline{F_{jk}}.
\end{align*}
We get,
\begin{align*}
    \Re{(dB_{jk})} & = \sigma_k\Re{(E_{jk})}-\sigma_j\Re{(F_{jk})},\\
    \Im{(dB_{jk})} & = \sigma_k\Im{(E_{jk})}-\sigma_j\Im{(F_{jk})},\\
    \Re{(dB_{kj})} & = -\sigma_j\Re{(E_{jk})}+\sigma_k\Re{(F_{jk})},\\
    \Im{(dB_{kj})} & = \sigma_j\Im{(E_{jk})}-\sigma_k\Im{(F_{jk})}.
\end{align*}
Computing the wedge of real parts
$$\Re{(dB_{jk})}\Re{(dB_{kj})} 
    = (\sigma_k^2-\sigma_j^2 )\Re{(E_{jk})}\Re{(F_{jk})},$$
and the imaginary parts
$$\Im{(dB_{jk})}\Im{(dB_{kj})}
    = (\sigma_k^2-\sigma_j^2)\Im{(F_{jk})}\Im{(E_{jk})},$$
we get that 
\begin{align*}
    & \Re{(dB_{jk})}\Im{(dB_{jk})}\Re{(dB_{kj})}\Im{(dB_{kj})}\\
    = & -(\sigma_k^2-\sigma_j^2)^2\Re{(\tu_j^*d\tu_k)}\Im{(\tu_j^*d\tu_k)}\Re{(\tv_j^*d\tv_k)}\Im{(\tv_j^*d\tv_k)}.
\end{align*}
Combining all these and ignoring the sign, we get the form 
\begin{align*}
    & \bigwedge_{j=1}^m \Re{(dB_{jj})}\bigwedge_{j,k=1, j\ne k}^m\Re{(dB_{jk})}\Im{(dB_{jk})} \\
    = & \bigwedge_{j=1}^md\sigma_j \prod_{1\le j < k \le m}^m(\sigma_k^2-\sigma_j^2)^2 \eta\bigwedge_{1\le j < k \le m}^m\Re{(\tu_j^*d\tu_k)}\Im{(\tu_j^*d\tu_k)}
\end{align*}
where 
$$\eta = \bigwedge_{1\le j < k \le m}^m \Re{(\tv_j^*d\tv_k)}\Im{(\tv_j^*d\tv_k)}.$$
We note that $\eta$ is a volume form of $U(m)/(U(1)^m)$. For the imaginary part of the diagonal entries, for $j,k \in \{1,...,m\}$ with $j\ne k$, let's look at the following
\begin{align*}
    & \Im{dB_{jj}}\wedge \Im{dB_{kk}}\\
    = & \sigma_j(\Im({\tu_j^*d\tu_j})-\Im({\tv_j^*d\tv_j}))\wedge \sigma_k(\Im({\tu_k^*d\tu_k})-\Im({\tv_k^*d\tv_k}))\\
    = & \sigma_j\sigma_k\Im({\tu_j^*d\tu_j})\Im({\tu_k^*d\tu_k})    + \kappa
\end{align*}
where $\kappa = \sigma_j\sigma_k[\Im({\tv_j^*d\tv_j})\Im({\tv_k^*d\tv_k}) -\Im({\tu_j^*d\tu_j})\Im({\tv_k^*d\tv_k})-\Im({\tu_k^*d\tu_k})\Im({\tv_j^*d\tv_j})] $. We note that $\eta$ is the volume form of the compact group $U(m)/U(1)^m$, so $\kappa \wedge \eta = 0$ (as $\eta$ already contains all the independent variables coming from $V$). So, $$\bigwedge_{j=1}^m \Im(dB_{jj}) = \prod_{j=1}^m\sigma_j \bigwedge_{j=1}^m\Im(\tu_j^*d\tu_j).$$

Therefore,
\begin{align*}
    & \bigwedge_{j,k=1}^m (\Re{(dB_{jk})} \Im{(dB_{jk}} )\\
    = & \prod_{j=1}^m\sigma_j\prod_{1\le j < k \le m}^m(\sigma_k^2-\sigma_j^2)^2 \eta \bigwedge_{1\le j < k \le m}^m\Re{(\tu_j^*d\tu_k)}\Im{(\tu_j^*d\tu_k)} \bigwedge_{j=1}^m \Im({\tu_j^*d\tu_j})\bigwedge_{j=1}^md\sigma_j
\end{align*}
Now, for $k\in \{1,2,...,m\}$ and $j\in \{m+1,m+2,...,n\}$, we have
$$dB_{jk}=(\tU_2d\tU_1\Sigma)_{j-m,k} = \sigma_k\tu_j^*d\tu_k.$$
Therefore,
\begin{align*}
    \bigwedge_{j=m+1}^n\bigwedge_{k=1}^m\Re(dB_{jk})\Im(dB_{jk}) & = \bigwedge_{j=m+1}^n\bigwedge_{k=1}^m\sigma_k^2\Re(\tu_j^*d\tu_k)\Im(\tu_j^*d\tu_k) \\
    & =\prod_{j=1}^m \sigma_k^{2(n-m)} \bigwedge_{j=m+1}^n\bigwedge_{k=1}^m\Re(\tu_j^*d\tu_k)\Im(\tu_j^*d\tu_k)
\end{align*}
Gathering all these and ignoring sign, we get that 
\begin{align*}
    \rho & =\bigwedge_{j=1}^n\bigwedge_{k=1}^m\Re(dB_{jk})\Im(dB_{jk})\\
    & = \prod_{j=1}^m\sigma_j\prod_{1\le j < k \le m}^m(\sigma_k^2-\sigma_j^2)^2 \prod_{j=1}^m \sigma_k^{2(n-m)}\eta\omega\bigwedge_{j=1}^md\sigma_j
\end{align*}
where $$\omega = \bigwedge_{1\le j < k \le m}^m\Re{(\tu_j^*d\tu_k)}\Im{(\tu_j^*d\tu_k)} \bigwedge_{j=1}^m \Im({\tu_j^*d\tu_j})\bigwedge_{j=m+1}^n\bigwedge_{k=1}^m\Re(\tu_j^*d\tu_k)\Im(\tu_j^*d\tu_k)$$
is a form independent of $\sigma_j$'s.
We have
\begin{align*}
    \int_{\Sph}d\mu & = \frac{1}{\text{Vol}(\mathbb{B}^{2mn})}\int_{\cS}d\nu\\
    & = \frac{c(m,n)}{\text{Vol}(\mathbb{B}^{2mn})}\int_{(U(m)/(U(1))^m) \times \mathcal B \times V_m^n(\C)}\rho\\
    & = \frac{c(m,n)}{\text{Vol}(\mathbb{B}^{2mn})}\int_{U(m)/(U(1))^m} \eta \int_{V_m^n(\C)}\omega \int_{\mathcal B}\prod_{j=1}^m\sigma_j\prod_{1\le j < k \le m}^m(\sigma_k^2-\sigma_j^2)^2 \prod_{j=1}^m \sigma_j^{2(n-m)}d\sigma_j
\end{align*}
where $c(m,n)$ is a constant due to the requirement that $\int_{\Sph}d\mu=1$. For the integral \\ $\int_{\mathcal{B}}\prod_{j=1}^m\sigma_j\prod_{1\le j < k \le m}^m(\sigma_k^2-\sigma_j^2)^2 \prod_{j=1}^m \sigma_j^{2(n-m)}d\sigma_j $, we change variables to $p_j=\sigma_j^2$ to get $$\int_{\mathcal{B}}\prod_{j=1}^m\sigma_j\prod_{1\le j < k \le m}^m(\sigma_k^2-\sigma_j^2)^2 \prod_{j=1}^m \sigma_j^{2(n-m)}d\sigma_j  = \frac{1}{2^m}\int_{T_m}\prod_{1\le j < k \le m}^m(p_k-p_j)^2 \prod_{j=1}^m p_j^{n-m}dp_j.$$
We note that for $v\in \mathbb{B}^{2mn}$, if $p_1,...,p_m$ are the eigenvalues of $\tr_2(vv^*)$, then $\frac{p_1}{\sum_{j=1}^m p_j},...,\frac{p_m}{\sum_{j=1}^m p_j}$ are the eigenvalues of $\tr_2(\frac{vv^*}{\|v\|^2})$. Therefore, for a function $f$ on $\Sph$ that depends only on the squares of the Schmidt coefficients $\sigma_j(u)$ of $u\in \Sph$ (in other words, depends only on $p_j(u)$'s), we get
\begin{align*}
    & \int_{\Sph}f(u)d\mu(u)\\
    = & \frac{1}{\text{Vol}(\mathbb{B}^{2mn})}\int_{\cS}f\left(\frac{v}{\|v\|}\right)d\nu(v)\\
    = & \frac{c(m,n)}{\text{Vol}(\mathbb{B}^{2mn})}\int_{(U(m)/(U(1))^m) \times \mathcal{B} \times V_m^n(\C)}f\left(\frac{v}{\|v\|}\right)\rho\\
    = & \frac{\frac{c(m,n)}{\text{Vol}(\mathbb{B}^{2mn})}\int_{(U(m)/(U(1))^m) \times \mathcal{B} \times V_m^n(\C)}f\left(\frac{v}{\|v\|}\right)\rho}{\frac{c(m,n)}{\text{Vol}(\mathbb{B}^{2mn})}\int_{(U(m)/(U(1))^m) \times \mathcal{B} \times V_m^n(\C)}\rho} \text{\quad(as the denominator is $\int_Sd\mu = 1$)}\\
    = & \frac{\int_{T_m} f\left(\frac{p_1(v)}{\sum_{k=1}^mp_k(v)},...,\frac{p_m(v)}{\sum_{k=1}^mp_k(v)}\right)\prod_{1\le j < k \le m}^m(p_k(v)-p_j(v))^2\prod_{j=1}^mp_j(v)^{n-m}dp_j(v)}{\int_{T_m}\prod_{1\le j < k \le m}^m(p_k(v)-p_j(v))^2 \prod_{j=1}^m p_j(v)^{n-m}dp_j(v)}.
\end{align*}
\end{proof}
\begin{theorem}
\label{averageth}
The expected value of the entropy of entanglement of all the pure states in $H_1\otimes H_2$ is
$$\langle E_{(m,n)} \rangle=\sum_{k=n+1}^{mn}\frac{1}{k}+\frac{m-1}{2n}.$$
\end{theorem}
\begin{proof} We base our proof on the general idea of the proof in \cite{sen}. 
The expected value of the entropy of entanglement is given by
\begin{align}
   & \int_{\Sph}E_{(m,n)}(u)d\mu(u) \nonumber \\
   = &\frac{\int_{T_m} \left(-\sum_{j=1}^m\frac{p_j}{\sum_{k=1}^mp_k}\ln \frac{p_j}{\sum_{k=1}^mp_k}\right)\prod_{1\le j < k \le m}^m(p_k-p_j)^2\prod_{j=1}^mp_j^{n-m}dp_j}{\int_{T_m}\prod_{1\le j < k \le m}^m(p_k-p_j)^2 \prod_{j=1}^m p_j^{n-m}dp_j}. \label{Th2eq1}     
\end{align}

We change variables to $(q_1,...,q_{m-1},r)$ such that $q_1+...+q_{m-1}+q_m=1$, and $q_k=rp_k$ for $k=1,2,...,m$ (i.e. $r=\frac{1}{\sum_{j=1}^mp_j}$), so $r \in [1,\infty)$. Then for $j\in \{1,...,m\}$ we have,
\begin{align*}
    & \frac{\partial q_k}{\partial p_j} = \delta_{jk}r \text{ for } k\in \{1,...,m-1\}\\
    & \frac{\partial r}{\partial p_j} = \frac{\partial }{\partial p_j}\left(\frac{1}{\sum_{k}p_k}\right) = \frac{-1}{(\sum_{k}p_k)^2} = -r^2
\end{align*}
We denote $F_m = \{(x_1,x_2,...,x_m) \in \R^m:x_1,x_2,...,x_m\ge 0 \text{ and } x_1+x_2+...+x_m=1\}$. So, the integral becomes
\begin{align}
    & \frac{\int_1^{\infty}\int_{F_m} \left(-\sum_{j=1}^mq_j\ln q_j\right)\prod_{1\le j < k \le m}^m(\frac{q_k}{r}-\frac{q_j}{r})^2\prod_{j=1}^m\left(\frac{q_j}{r}\right)^{n-m}\frac{1}{r^{m+1}}dr\prod_{j=1}^{m-1}dq_j}{\int_1^{\infty}\int_{F_m}\prod_{1\le j < k \le m}^m(\frac{q_k}{r}-\frac{q_j}{r})^2 \prod_{j=1}^m \left(\frac{q_j}{r}\right)^{n-m}\frac{1}{r^{m+1}}dr\prod_{j=1}^{m-1}dq_j} \nonumber \\
    = & \frac{\int_1^{\infty}\frac{dr}{r^{mn+1}}\int_{F_m} \left(-\sum_{j=1}^mq_j\ln q_j\right)\prod_{1\le j < k \le m}^m(q_k-q_j)^2\prod_{j=1}^mq_j^{n-m} \prod_{j=1}^{m-1}dq_j}{\int_1^{\infty}\frac{dr}{r^{mn+1}}\int_{F_m}\prod_{1\le j < k \le m}^m(q_k-q_j)^2 \prod_{j=1}^m q_j^{n-m}\prod_{j=1}^{m-1}dq_j} \nonumber \\    
    = & \frac{\int_{F_m} \left(-\sum_{j=1}^mq_j\ln q_j\right)\prod_{1\le j < k \le m}^m(q_k-q_j)^2\prod_{j=1}^mq_j^{n-m}\prod_{j=1}^{m-1} dq_j}{\int_{F_m}\prod_{1\le j < k \le m}^m(q_k-q_j)^2 \prod_{j=1}^m q_j^{n-m}\prod_{j=1}^{m-1}dq_j} \label{Convert2qs}
\end{align}
Let $x_k=tq_k$ for $k=1,2,...,m$ with $t\in [0,\infty)$, so that $x_1+x_2+...+x_m=t$ and $x_k \in [0,\infty)$. Then, for $j\in \{1,2,...,m\}$, we have
\begin{align*}
    & \frac{\partial q_k}{\partial x_j} = \delta_{jk}t \text{ for } k\in \{1,2,...,m-1\}\\
    & \frac{\partial t}{\partial x_j} = 1
\end{align*}
Using the determinant of the Jacobian, we get
\begin{align}
    dq_1\wedge...\wedge dq_{m-1} \wedge dt & = \frac{1}{t^{m-1}}dx_1\wedge...\wedge dx_{m-1} \wedge dx_{m} \nonumber\\
    & = \frac{1}{t^{m-1}}dx_1\wedge...\wedge dx_{m-1} \wedge d(x_1+...+x_{m-1}+x_{m})\nonumber\\
    & = \frac{1}{t^{m-1}}dx_1\wedge...\wedge dx_{m-1} \wedge dt \label{Th2qtox}
\end{align}
Now, we use the gamma function $\Gamma(z) = \int_0^{\infty} y^{z-1}e^{-y}dy$ and the derivative of the gamma function $\Gamma'(z) = \int_0^{\infty}y^{z-1}e^{-y}\ln y dy$. In particular,
\begin{align*}
    & \int_0^{\infty}t^{mn}e^{-t}dt = \Gamma(mn+1) = (mn)!,\\
    & \int_0^{\infty}t^{mn}e^{-t}\ln tdt = \Gamma'(mn+1) = (mn)!\psi(mn+1),
\end{align*}
where $\psi(N+1) = -\gamma + \sum_{k=1}^{N}\frac{1}{k}$ and $\gamma$ is the Euler constant. We have,
\begin{align}
    & \frac{\int_{F_m}\left(-\sum_{j=1}^mq_j\ln q_j\right)\prod_{1\le j < k \le m}^m(q_k-q_j)^2\prod_{j=1}^mq_j^{n-m}\prod_{j=1}^{m-1} dq_j}{\int_{F_m}\prod_{1\le j < k \le m}^m(q_k-q_j)^2 \prod_{j=1}^m q_j^{n-m}\prod_{j=1}^{m-1}dq_j} \nonumber \\
    = & \frac{(mn-1)!}{(mn)!}\frac{\int_0^{\infty}t^{mn}e^{-t}dt\int_{F_m}\left(-\sum_{j=1}^mq_j\ln q_j\right)\prod_{1\le j < k \le m}^m(q_k-q_j)^2\prod_{j=1}^mq_j^{n-m}\prod_{j=1}^{m-1} dq_j}{\int_0^{\infty}t^{mn-1}e^{-t}dt\int_{F_m}\prod_{1\le j < k \le m}^m(q_k-q_j)^2 \prod_{j=1}^m q_j^{n-m}\prod_{j=1}^{m-1}dq_j} \nonumber \\
    = & \frac{1}{mn}\frac{\int_{[0,\infty)^m}\left(-\sum_{j=1}^m\left(\frac{x_j}{t}\right)\ln \left(\frac{x_j}{t}\right)\right)\prod_{1\le j < k \le m}^m\left(\frac{x_k}{t}-\frac{x_j}{t}\right)^2\prod_{j=1}^m\left(\frac{x_j}{t}\right)^{n-m}e^{-t}t^{mn}\frac{dt}{t^{m-1}}\prod_{j=1}^{m-1} dx_j}{\int_{[0,\infty)^m}\prod_{1\le j < k \le m}^m\left(\frac{x_k}{t}-\frac{x_j}{t}\right)^2 \prod_{j=1}^m \left(\frac{x_j}{t}\right)^{n-m}e^{-t}t^{mn-1}\frac{dt}{t^{m-1}}\prod_{j=1}^{m-1}dx_j} \nonumber \\
    = & \frac{1}{mn}\frac{\int_{[0,\infty)^m}\left(-\sum_{j=1}^mx_j\ln\left(\frac{x_j}{t}\right)\right)\prod_{1\le j < k \le m}^m(x_k-x_j)^2\prod_{j=1}^mx_j^{n-m}e^{-t}dt\prod_{j=1}^{m-1} dx_j}{\int_{[0,\infty)^m}\prod_{1\le j < k \le m}^m(x_k-x_j)^2 \prod_{j=1}^m x_j^{n-m}e^{-t}dt\prod_{j=1}^{m-1}dx_j}\nonumber \\
    = & \frac{\int_{[0,\infty)^m}t\ln t\prod_{1\le j < k \le m}^m(x_k-x_j)^2\prod_{j=1}^mx_j^{n-m}e^{-t}dt\prod_{j=1}^{m-1} dx_j}{mn\int_{[0,\infty)^m}\prod_{1\le j < k \le m}^m(x_k-x_j)^2 \prod_{j=1}^m x_j^{n-m}e^{-t}dt\prod_{j=1}^{m-1}dx_j} \nonumber \\
    & \qquad - \frac{\int_{[0,\infty)^m}\left(\sum_{j=1}^mx_j\ln x_j\right)\prod_{1\le j < k \le m}^m(x_k-x_j)^2\prod_{j=1}^mx_j^{n-m}e^{-t}dt\prod_{j=1}^{m-1} dx_j}{mn\int_{[0,\infty)^m}\prod_{1\le j < k \le m}^m(x_k-x_j)^2 \prod_{j=1}^m x_j^{n-m}e^{-t}dt\prod_{j=1}^{m-1}dx_j} \label{I1I2split}
\end{align}
Let the first and the second integral in equation (\ref{I1I2split}) denoted by $I_1$ and $I_2$ respectively. We have
\begin{align}
    I_1
    =& \frac{\int_{[0,\infty)^m}t\ln t\prod_{1\le j < k \le m}^m(x_k-x_j)^2\prod_{j=1}^mx_j^{n-m}e^{-t}dt\prod_{j=1}^{m-1} dx_j}{mn\int_{[0,\infty)^m}\prod_{1\le j < k \le m}^m(x_k-x_j)^2 \prod_{j=1}^m x_j^{n-m}e^{-t}dt\prod_{j=1}^{m-1}dx_j}\nonumber\\
    = & \frac{\int_0^{\infty}\int_{F_m}t\ln t\prod_{1\le j < k \le m}^m(tq_k-tq_j)^2\prod_{j=1}^m(tq_j)^{n-m}e^{-t}t^{m-1}dt\prod_{j=1}^{m-1} dq_j}{mn\int_0^{\infty}\int_{F_m}\prod_{1\le j < k \le m}^m(tq_k-tq_j)^2 \prod_{j=1}^m (tq_j)^{n-m}e^{-t}t^{m-1}dt\prod_{j=1}^{m-1}dq_j}\nonumber\\
    = & \frac{\int_0^{\infty}t^{mn}\ln te^{-t}dt\int_{F_m}\prod_{1\le j < k \le m}^m(q_k-q_j)^2\prod_{j=1}^mq_j^{n-m}\prod_{j=1}^{m-1} dq_j}{mn\int_0^{\infty}t^{mn-1}e^{-t}dt\int_{F_m}\prod_{1\le j < k \le m}^m(q_k-q_j)^2 \prod_{j=1}^m q_j^{n-m}\prod_{j=1}^{m-1}dq_j}\nonumber\\
    = & \frac{\Gamma'(mn+1)}{mn\Gamma(mn)}\nonumber\\
    = & \psi(mn+1) \label{I1final}
\end{align}
Using equation (\ref{Th2qtox}), the second integral in equation (\ref{I1I2split}) becomes
\begin{align}
    I_2= & \frac{\int_{[0,\infty)^m}\left(\sum_{j=1}^mx_j\ln x_j\right)\prod_{1\le j < k \le m}^m(x_k-x_j)^2\prod_{j=1}^mx_j^{n-m}e^{-t}dt\prod_{j=1}^{m-1} dx_j}{mn\int_{[0,\infty)^m}\prod_{1\le j < k \le m}^m(x_k-x_j)^2 \prod_{j=1}^m x_j^{n-m}e^{-t}dt\prod_{j=1}^{m-1}dx_j} \nonumber\\
    = & \frac{\sum_{l=1}^m\int_{[0,\infty)^m}x_l\ln x_l\prod_{1\le j < k \le m}^m(x_k-x_j)^2\prod_{j=1}^mx_j^{n-m}e^{-(x_1+...+x_m)}\prod_{j=1}^m dx_j}{mn\int_{[0,\infty)^m}\prod_{1\le j < k \le m}^m(x_k-x_j)^2 \prod_{j=1}^m x_j^{n-m}e^{-(x_1+...+x_m)}\prod_{j=1}^mdx_j} \label{ToughestInt}
\end{align}
We observe that the van der Monde determinant
$$\Delta(x_1,...,x_m) = \prod_{1\le j < k \le m}^m(x_k-x_j) = \det\begin{pmatrix} 1 & \dots & 1\\ x_1 & \dots & x_m\\ \vdots & \ddots & \vdots\\ x_1^{m-1} &  \dots & x_m^{m-1} \end{pmatrix}.$$
As determinant remains unchanged after applying elementary row operations, we see that
$$\prod_{1\le j < k \le m}^m(x_k-x_j) = \det\begin{pmatrix} f_0(x_1) & \dots & f_0(x_m) \\ f_1(x_1) & \dots & f_1(x_m)\\ \vdots & \ddots & \vdots\\ f_{m-1}(x_1) &  \dots & f_{m-1}(x_m) \end{pmatrix} = \det_{m\times m}(f_{j-1}(x_k)),$$
for any set of polynomials $\{f_0,f_1,...,f_{m-1}\}$ with $f_0=1$ and $f_k$'s are monic with $\deg(f_k)=k$ for $k=1,...,m-1$. We choose these polynomials cleverly for the $\Delta(x_1,...,x_m)^2$ appearing in the numerator and the denominator in equation (\ref{ToughestInt}). We use a class of orthogonal polynomials, generalized Laguerre polynomials, given by $$L_k^{(\alpha)}(x) =  \frac{x^{-\alpha}e^x}{k!}\frac{d^k}{dx^k}(e^{-x}x^{k+\alpha}),$$
where $\alpha \in \R$ and $k\in \N \cup \{0\}$. These polynomials satisfy the orthogonality relation,
\begin{equation}
    \int_0^{\infty}x^{\alpha}e^{-x}L_k^{(\alpha)}(x)L_j^{(\alpha)}(x)dx = \frac{\Gamma(k+\alpha+1)}{k!}\delta_{jk}. \label{ortho}
\end{equation}
Using these polynomials,
$$\Delta(x_1,...,x_m) = \det_{m\times m}\left(L_{j-1}^{(\alpha)}(x_k)\right) =\sum_{\sigma \in S_m}\text{sgn}(\sigma)\prod_{k=1}^{m}L_{\sigma(k-1)}^{(n-m)}(x_k).$$
The integral in equation (\ref{ToughestInt}) becomes
\begin{align}
    I_2 =& \frac{\sum_{l=1}^m\int_{[0,\infty)^m}x_l\ln x_l\prod_{1\le j < k \le m}^m(x_k-x_j)^2\prod_{j=1}^mx_j^{n-m}e^{-(x_1+...+x_m)}\prod_{j=1}^m dx_j}{mn\int_{[0,\infty)^m}\prod_{1\le j < k \le m}^m(x_k-x_j)^2 \prod_{j=1}^m x_j^{n-m}e^{-(x_1+...+x_m)}\prod_{j=1}^mdx_j} \nonumber \\
    = & \frac{\sum_{l=1}^m\int_{[0,\infty)^m}x_l\ln x_l\Delta(x_1,...,x_m)\Delta(x_1,...,x_m)\prod_{j=1}^mx_j^{n-m}e^{-x_j}dx_j}{mn\int_{[0,\infty)^m}\Delta(x_1,...,x_m)\Delta(x_1,...,x_m)\prod_{j=1}^mx_j^{n-m}e^{-x_j}dx_j} \nonumber \\
    = & \frac{\sum_{l=1}^m\int_{[0,\infty)^m}x_l\ln x_l\sum_{\sigma,\tau \in S_m}\text{sgn}(\sigma)\text{sgn}(\tau)\prod_{k=1}^{m}L_{\sigma(k-1)}^{(n-m)}(x_k)L_{\tau(k-1)}^{(n-m)}(x_k)\prod_{j=1}^mx_j^{n-m}e^{-x_j}dx_j}{mn\int_{[0,\infty)^m}\sum_{\sigma,\tau \in S_m}\text{sgn}(\sigma)\text{sgn}(\tau)\prod_{k=1}^mL_{\sigma(k-1)}^{(n-m)}(x_k)L_{\tau(k-1)}^{(n-m)}(x_k)\prod_{j=1}^mx_j^{n-m}e^{-x_j}dx_j} \nonumber \\
    = & \frac{\sum_{l=1}^m\sum_{\sigma,\tau \in S_m}\text{sgn}(\sigma)\text{sgn}(\tau)\int_{[0,\infty)^m}x_l\ln x_l\prod_{k=1}^mL_{\sigma(k-1)}^{(n-m)}(x_k)L_{\tau(k-1)}^{(n-m)}(x_k)x_k^{n-m}e^{-x_k}dx_k}{mn\sum_{\sigma,\tau \in S_m}\text{sgn}(\sigma)\text{sgn}(\tau)\int_{[0,\infty)^m}\prod_{k=1}^mL_{\sigma(k-1)}^{(n-m)}(x_k)L_{\tau(k-1)}^{(n-m)}(x_k)x_k^{n-m}e^{-x_k}dx_k} \nonumber \\
    = & \frac{\sum_{l=1}^m\sum_{\sigma,\tau \in S_m}\text{sgn}(\sigma)\text{sgn}(\tau)\prod_{k=1}^m\int_0^{\infty}(x_k\ln x_k)^{\delta_{lk}} L_{\sigma(k-1)}^{(n-m)}(x_k)L_{\tau(k-1)}^{(n-m)}(x_k)x_k^{n-m}e^{-x_k}dx_k}{mn\sum_{\sigma,\tau \in S_m}\text{sgn}(\sigma)\text{sgn}(\tau)\prod_{k=1}^m\int_0^{\infty}L_{\sigma(k-1)}^{(n-m)}(x_k)L_{\tau(k-1)}^{(n-m)}(x_k)x_k^{n-m}e^{-x_k}dx_k} \nonumber \\
    = & \frac{\sum_{l=1}^m\sum_{\sigma \in S_m}\prod_{k=1}^m\int_0^{\infty}(x_k\ln x_k)^{\delta_{lk}} [L_{\sigma(k-1)}^{(n-m)}(x_k)]^2x_k^{n-m}e^{-x_k}dx_k}{mn\sum_{\sigma \in S_m}\prod_{k=1}^m\int_0^{\infty}[L_{\sigma(k-1)}^{(n-m)}(x_k)]^2x_k^{n-m}e^{-x_k}dx_k} \text{(using the orthogonal property)} \nonumber \\
    = & \frac{\sum_{l=1}^m\sum_{\sigma \in S_m}\int_0^{\infty}x_l^{n-m+1}\ln x_l[L_{\sigma(l-1)}^{(n-m)}(x_l)]^2e^{-x_l}dx_l\prod_{k=1,k\ne l}^m\int_0^{\infty} [L_{\sigma(k-1)}^{(n-m)}(x_k)]^2x_k^{n-m}e^{-x_k}dx_k}{mn\sum_{\sigma \in S_m}\prod_{k=1}^m\int_0^{\infty}[L_{\sigma(k-1)}^{(n-m)}(x_k)]^2x_k^{n-m}e^{-x_k}dx_k}\nonumber \\
    = & \frac{\sum_{l=1}^m\sum_{\sigma \in S_m}\int_0^{\infty}x^{n-m+1}\ln x[L_{\sigma(l-1)}^{(n-m)}(x)]^2e^{-x}dx\prod_{k=1,k\ne l}^m\int_0^{\infty} [L_{\sigma(k-1)}^{(n-m)}(x)]^2x^{n-m}e^{-x}dx}{mn\sum_{\sigma \in S_m}\prod_{k=1}^m\int_0^{\infty}[L_{\sigma(k-1)}^{(n-m)}(x)]^2x^{n-m}e^{-x}dx} \nonumber\\
    = & \frac{\sum_{l=1}^m(m-1)!\sum_{k=1}^m\int_0^{\infty}x^{n-m+1}\ln x[L_{l-1}^{(n-m)}(x)]^2e^{-x}dx\prod_{k=1,k\ne l}^m\int_0^{\infty} [L_{k-1}^{(n-m)}(x)]^2x^{n-m}e^{-x}dx}{mn(m)!\prod_{j=1}^m\int_0^{\infty}[L_{j-1}^{(n-m)}(x)]^2x^{n-m}e^{-x}dx}
    \nonumber\\
    = & \sum_{l=1}^m\sum_{k=1}^m\frac{\int_0^{\infty}x^{n-m+1}\ln x[L_{l-1}^{(n-m)}(x)]^2e^{-x}dx}{m^2n\int_0^{\infty}[L_{l-1}^{(n-m)}(x)]^2x^{n-m}e^{-x}dx}
    \nonumber\\
    = & m\sum_{l=1}^m\frac{\int_0^{\infty}x^{n-m+1}\ln x[L_{l-1}^{(n-m)}(x)]^2e^{-x}dx}{m^2n\int_0^{\infty}[L_{l-1}^{(n-m)}(x)]^2x^{n-m}e^{-x}dx}
    \nonumber\\
    = & \frac{1}{mn}\sum_{k=0}^{m-1}\frac{\int_0^{\infty}x^{n-m+1}\ln x[L_{k}^{(n-m)}(x)]^2e^{-x}dx}{\int_0^{\infty}[L_{k}^{(n-m)}(x)]^2x^{n-m}e^{-x}dx}\label{I2penulimate}
\end{align}
Let $I_k^{(\alpha)} = \int_0^{\infty}x^{\alpha+1}\ln x[L_{k}^{(\alpha)}(x)]^2e^{-x}dx$ and $J_k(\alpha) = \int_0^{\infty}x^{\alpha+1}[L_{k}^{(\alpha)}(x)]^2e^{-x}dx$. By properties of Laguerre polynomials, we have 
\begin{equation}
    J_k(\alpha) = \frac{(2k+\alpha+1)\Gamma(k+\alpha+1)}{k!}. \label{jkdef}
\end{equation}
Now,
\begin{align}
    & \frac{d}{d\alpha}J_k(\alpha) = \int_0^{\infty}x^{\alpha+1}\ln x[L_{k}^{(\alpha)}(x)]^2e^{-x}dx + 2\int_0^{\infty}x^{\alpha+1}L_{k}^{(\alpha)}(x)\frac{d L_{k}^{(\alpha)}(x)}{d\alpha}e^{-x}dx \nonumber\\
    \implies & I_k^{(n-m)} = \left[\frac{d}{d\alpha}J_k(\alpha) -  2\int_0^{\infty}x^{\alpha+1}L_{k}^{(\alpha)}(x)\frac{d L_{k}^{(\alpha)}(x)}{d\alpha}e^{-x}dx\right]_{\alpha = n-m} \label{Ik}
\end{align}
Using equation (\ref{jkdef}), we get
\begin{align}
    \frac{d}{d\alpha}J_k(\alpha) & = \frac{d}{d\alpha}\left( \frac{(2k+\alpha+1)\Gamma(k+\alpha+1)}{k!}\right) \nonumber \\
    & = \frac{\Gamma(k+\alpha+1)}{k!} + \frac{2k+\alpha+1}{k!}\frac{d\Gamma(k+\alpha+1)}{d\alpha} \nonumber\\
    & = \frac{\Gamma(k+\alpha+1)}{k!} + \frac{2k+\alpha+1}{k!}\Gamma(k+\alpha+1)\psi(k+\alpha+1)\nonumber\\
    & = \frac{\Gamma(k+\alpha+1)}{k!}[1+(2k+\alpha+1)\psi(k+\alpha+1)] \label{difJk}
\end{align}
We use the property $L_k^{(\alpha)}(x) = L_k^{(\alpha+1)}(x)-L_{k-1}^{(\alpha+1)}(x)$ and $\frac{d L_{k}^{(\alpha)}(x)}{d\alpha} = \sum_{j=0}^{k-1}\frac{L_j^{\alpha}(x)}{k-j}$ to compute 
\begin{align}
    & \int_0^{\infty}x^{\alpha+1}L_{k}^{(\alpha)}(x)\frac{d L_{k}^{(\alpha)}(x)}{d\alpha}e^{-x}dx \nonumber\\
    = & \int_0^{\infty}x^{\alpha+1}L_{k}^{(\alpha)}(x)\sum_{j=0}^{k-1}\frac{L_j^{\alpha}(x)}{k-j}e^{-x}dx\nonumber\\
    = & \sum_{j=0}^{k-1}\frac{1}{k-j}\int_0^{\infty}x^{\alpha+1}\left(L_k^{(\alpha+1)}(x)-L_{k-1}^{(\alpha+1)}(x)\right)\left(L_j^{(\alpha+1)}(x)-L_{j-1}^{(\alpha+1)}(x)\right)e^{-x}dx\nonumber\\
    = & -\int_0^{\infty}x^{\alpha+1}[L_{k-1}^{(\alpha+1)}(x)]^2e^{-x}dx \nonumber \\
    = & -\frac{\Gamma(k+\alpha+1)}{(k-1)!} \label{2ndTermIk}
\end{align}
Using (\ref{difJk}) and (\ref{2ndTermIk}) in (\ref{Ik}), we get
\begin{align}
    I_k^{(n-m)} & = \left[ \frac{\Gamma(k+\alpha+1)}{k!}[1+(2k+\alpha+1)\psi(k+\alpha+1)] + 2\frac{\Gamma(k+\alpha+1)}{(k-1)!}\right]_{\alpha=n-m}\nonumber \\
    & = \left[\frac{\Gamma(k+\alpha+1)}{k!}(1+2k+(2k+\alpha+1)\psi(k+\alpha+1))\right]_{\alpha=n-m} \nonumber \\
    & = \frac{\Gamma(k+n-m+1)}{k!}[1+2k+(2k+n-m+1)\psi(k+n-m+1)] \label{Ikfinal}
\end{align}
Using equation (\ref{ortho}) and (\ref{Ikfinal}) in equation (\ref{I2penulimate}), we get
\begin{align}
    I_2 = & \frac{1}{mn}\sum_{k=0}^{m-1}[1+2k+(2k+n-m+1)\psi(k+n-m+1)] \nonumber\\
    = & \frac{1}{mn}\sum_{0}^{m-1}(1+2k)+ \frac{1}{mn}\sum_{k=0}^{m-1}\left[(2k+n-m+1)\left(-\gamma+ \sum_{r=1}^{n-m+k}\frac{1}{r}\right)\right] \nonumber \\
    = & \frac{m+m(m-1)}{mn} -\gamma\frac{1}{mn}\sum_{k=0}^{m-1}(2k+n-m+1)+\frac{1}{mn}\sum_{k=0}^{m-1}\sum_{r=1}^{n-m+k}\frac{2k+n-m+1}{r} \nonumber\\
    = & \frac{m}{n}-\gamma+\left[mn+\frac{mn}{2}+...+\frac{mn}{n-m}+\frac{mn-(n-m+1)}{n-m+1}+...\right. \nonumber\\ &\left. \quad\quad+\frac{mn-(n-m+1)-(n-m+1+2)...-((n-m-1)+2(m-1)))}{n-m+m-1} \right]\times \frac{1}{mn}\nonumber\\
    = & \frac{m}{n}-\gamma + \sum_{k=1}^{n-1}\frac{1}{k}-\frac{1}{mn}\left[\frac{n-m+1}{n-m+1}+....+\frac{(m-1)(n-m)+(m-1)^2}{n-m+m-1}\right]\nonumber\\
    = & \frac{m}{n}-\gamma + \sum_{k=1}^{n-1}\frac{1}{k}-\frac{1}{mn}\left[1+....+(m-1)\right]\nonumber\\
    = & -\gamma + \sum_{k=1}^{mn}\frac{1}{k}-\sum_{k=n+1}^{mn}\frac{1}{k}+\frac{m-1}{2n}\nonumber\\
    = &\psi(mn+1)-\sum_{k=n+1}^{mn}\frac{1}{k}+\frac{m-1}{2n}
    \label{I2final}
\end{align}
Using equation (\ref{I1final}), (\ref{I2final}), (\ref{I1I2split}) and (\ref{Convert2qs}), we get the expected value of Entropy of Entanglement over the pure states is
$$\sum_{k=n+1}^{mn}\frac{1}{k}-\frac{m-1}{2n}.$$
\end{proof}

\subsection{Proof of Theorem \ref{mainth}} \label{mainproof}
For $j\in \{ 1,2\}$, let $(L_j,h_j)$ be a holomorphic hermitian line bundle on a compact 
K\"ahler manifold  $(M_j,\omega_j)$ of complex dimension $d_j\ge 1$
such that the curvature of the Chern connection on $L_j$ 
is $-i\omega_j$, and $d_1\le d_2$. For $N\in\N$, 
the Hilbert spaces $H_1$ (of dimension $m=m(N)$) and $H_2$ (of dimension $n=n(N)$) will be $H^0(M_1,L_1^N)$ and $H^0(M_2,L_2^N)$. Let $N\to\infty$. 
We have \cite[sec. 4.1.1.]{mmbook}: 
\begin{equation}
\label{dimas}
m=m(N)=
\beta_1N^{d_1}+
\gamma_1 N^{d_1-1}+
O(N^{d_1-2})
\end{equation}
\begin{equation}
\label{dimas2}
n=n(N)=
\beta_2N^{d_2}+
\gamma_2 N^{d_2-1}+
O(N^{d_2-2})
\end{equation}
where   
$$
\beta_j=\int_{M_j} \frac{c_1(L_j)^{d_j}}{d_j!}
$$
$$
\gamma_j=\frac{1}{2}\int_{M_j} \frac{c_1(TM_j)c_1(L_j)^{d_j-1}}{(d_j-1)!}
$$
for $j\in \{ 1,2\}$.

We notice that $m\le n$ for large  $N$. 
By Theorem \ref{averageth}, the average entanglement entropy $\langle E_N\rangle$ over all the pure states in $H^0(M_1,L_1^N)\otimes H^0(M_2,L_2^N)$ equals 
\begin{equation}
\label{entr}
\Bigl (\sum _{k=n+1}^{mn}\frac{1}{k}\Bigr ) - \frac{m-1}{2n}. 
\end{equation}
To figure out the asymptotics of (\ref{entr}), 
we apply the Euler-Maclaurin formula to $f(x)=\frac{1}{x}$, to conclude that 
$$
\sum _{k=n+1}^{mn}\frac{1}{k}=\int_n^{mn}\frac{1}{x}dx+\frac{f(mn)-f(n)}{2}+
\sum_{k=1}^{[\frac{p}{2}]} \frac{B_{2k}}{(2k)!}(f^{(2k-1)}(mn)-f^{(2k-1)}(n))+R_p
$$
where $B_{2k}$ are the Bernoulli numbers, in particular $B_2=\frac{1}{6}$, and for the remainder we have the estimate 
$$
|R_p|\le \frac{2\zeta(p)}{(2\pi )^p}\int _n^{mn} |f^{(p)}(x)|dx. 
$$
Therefore, $\langle E_N\rangle$ becomes 
\begin{equation} \label{AvEntLast}
\ln m +\frac{1}{2mn}-\frac{m}{2n}+\sum_{k=1}^{[\frac{p}{2}]} \frac{B_{2k}}{(2k)!}(f^{(2k-1)}(mn)-f^{(2k-1)}(n))+R_p.    \end{equation}

In (\ref{AvEntLast}), let us set $p=2$ in the part 
$$
\frac{1}{2mn}+\sum_{k=1}^{[\frac{p}{2}]} \frac{B_{2k}}{(2k)!}(f^{(2k-1)}(mn)-f^{(2k-1)}(n))+R_p
$$ 
and we can now conclude that this part is 
$O(N^{-2d_2})$, because  
$$
|R_2|\le \frac{\zeta(2)}{2\pi ^2}\int _n^{mn} |f''(x)|dx= \frac{1}{12}(\frac{1}{n^2}-\frac{1}{m^2n^2}), 
$$
$$
f'(mn)-f'(n)=-\frac{1}{m^2n^2}+\frac{1}{n^2}
$$
and by (\ref{dimas},\ref{dimas2}). 

It remains to consider the term  $\ln m -\frac{m}{2n}$ in (\ref{AvEntLast}). By (\ref{dimas}) 
$$
\ln m=\ln (\beta_1N^{d_1}(1+\frac{\gamma_1}{\beta_1}\frac{1}{N}+O(\frac{1}{N^2}))\sim 
\ln\beta_1+d_1\ln N+\frac{\gamma_1}{\beta_1}\frac{1}{N}+O(\frac{1}{N^2}).
$$
If $d_1=d_2$, then by  (\ref{dimas},\ref{dimas2}), we get 
$$
\frac{m}{2n}=\frac{\beta_1(1+\frac{\gamma_1}{\beta_1}\frac{1}{N}+ O(\frac{1}{N^2})) }
{2\beta_2(1+\frac{\gamma_2}{\beta_2}\frac{1}{N}+ O(\frac{1}{N^2})) }\sim 
\frac{\beta_1}{2\beta_2}\Bigl ( 1+(\frac{\gamma_1}{\beta_1}-\frac{\gamma_2}{\beta_2} ) \frac{1}{N} \Bigr )+ O(\frac{1}{N^2}).
$$
Similarly, if $d_1=d_2-1$, then 
$$
\frac{m}{2n}=\frac{\beta_1(1+\frac{\gamma_1}{\beta_1}\frac{1}{N}+ O(\frac{1}{N^2})) }
{2\beta_2N(1+\frac{\gamma_2}{\beta_2}\frac{1}{N}+ O(\frac{1}{N^2})) }\sim 
\frac{\beta_1}{2\beta_2}\frac{1}{N} + O(\frac{1}{N^2}),
$$
and if $d_1-d_2\le -2$, then 
$$
\frac{m}{2n}\sim O(\frac{1}{N^2}). 
$$
The statement of the theorem follows.

\end{document}